\documentclass[a4paper,12pt,reqno]{amsart}
\usepackage[T1]{fontenc}
\usepackage[utf8]{inputenc}
\usepackage{amssymb,bbm,mathrsfs}
\usepackage[margin=1in]{geometry}
\usepackage[inline]{enumitem}
\usepackage{graphicx,xcolor}
\usepackage[bookmarksdepth=2]{hyperref}

\numberwithin{equation}{section}

\newtheorem{theorem}{Theorem}[section]
\newtheorem{lemma}[theorem]{Lemma}

\theoremstyle{definition}

\newtheorem{example}[theorem]{Example}
\newtheorem{remark}[theorem]{Remark}
\newtheorem{assumption}{Assumption}

\newcommand{\A}{\mathfrak{A}}
\newcommand{\X}{\mathfrak{X}}
\newcommand{\dom}{\mathfrak{D}}

\newcommand{\pr}{\mathbf{P}}
\newcommand{\ex}{\mathbf{E}}

\newcommand{\R}{\mathbf{R}}

\newcommand{\ind}{\mathbf{1}}
\newcommand{\sub}{\subseteq}
\newcommand{\ph}{\varphi}

\newcommand{\eps}{\varepsilon}

\newcommand{\expr}[1]{\left( #1 \right)}

\newcommand{\thet}{\vartheta}

\newcommand{\ignore}[1]{}

\newcommand{\dwalk}{d_{\mathrm{w}}}
\newcommand{\cbhi}{C_{\mathrm{BHI}}}
\newcommand{\cnu}{C_{\text{\rm L\'evy}}}
\newcommand{\cnuint}{C_{\text{\rm L\'evy-int}}}
\newcommand{\cnuinf}{C_{\text{\rm L\'evy-inf}}}
\newcommand{\cro}{\rho}
\newcommand{\cgreen}{C_{\mathrm{green}}}
\newcommand{\ctau}{C_{\mathrm{exit}}}
\newcommand{\caux}{C_{\tau}}

\usepackage{ifthen}
\newcommand{\formula}[2][nolabel]%
{%
 \ifthenelse{\equal{#1}{nolabel}}%
 {\begin{align*} #2 \end{align*}}%
 {%
  \ifthenelse{\equal{#1}{}}%
  {\begin{align} #2 \end{align}}%
  {\begin{align} \label{#1} #2 \end{align}}%
 }%
}

\begin{document}

%
%

\title{Martin kernels for Markov processes with jumps}
\author{Tomasz Juszczyszyn, Mateusz Kwa\'snicki}
\thanks{Work supported by the Polish National Science Centre (NCN) grant no. 2011/03/D/ST1/00311}
\address{Tomasz Juszczyszyn, Mateusz Kwa\'snicki \\ Department of Mathematics \\ Wroc\l aw University of Technology \\ ul. Wybrze\.ze Wyspia\'nskiego 27 \\ 50-370 Wroc\l aw, Poland}
\email{mateusz.kwasnicki@pwr.edu.pl}
\date{\today}

\begin{abstract}
We prove existence of boundary limits of ratios of positive harmonic functions for a wide class of Markov processes with jumps and irregular domains, in the context of general metric measure spaces. As a corollary, we prove uniqueness of the Martin kernel at each boundary point, that is, we identify the Martin boundary with the topological boundary. We also prove a Martin representation theorem for harmonic functions. Examples covered by our results include: strictly stable L\'evy processes in $\R^d$ with positive continuous density of the L\'evy measure; stable-like processes in $\R^d$ and in domains; and stable-like subordinate diffusions in metric measure spaces.
\end{abstract}

\maketitle

%
%

\section{Introduction}
\label{sec:intro}

The purpose of this article is to study boundary limits of ratios of positive functions which are harmonic in an arbitrary open set with respect to a Markov process with jumps. The proof of our main result, Theorem~\ref{thm:martin}, relies on the \emph{boundary Harnack inequality} for Markov processes with jumps, proved recently in~\cite{bib:bkk15}, and the oscillation reduction argument, developed in~\cite{bib:b97} and~\cite{bib:bkk08}. As an application, we obtain Martin representation of harmonic functions in Theorem~\ref{thm:repr}.

To explain the motivation for our research, we begin with a discussion of the classical case, where harmonicity has its usual meaning: $f$ is harmonic in an open set $D$ if $\Delta f = 0$ in $D$. The boundary Harnack inequality is a statement about positive harmonic functions in an open set, which are equal to zero on a part of the boundary. The result states that if $D$ is regular enough (for example, a Lipschitz domain), $x_0$ is a boundary point of $D$, $f$ and $g$ are positive and harmonic in $D$, and both $f$ and $g$ converge to $0$ on $\partial D \cap B(x_0, R)$, then for every $r \in (0, R)$ the ratio $f / g$ has bounded \emph{relative oscillation} in $D \cap B(x_0, r)$:
\formula[eq:ro]{
 \sup_{x \in D \cap B(x_0, r)} \frac{f(x)}{g(x)} & \le c \inf_{x \in D \cap B(x_0, r)} \frac{f(x)}{g(x)} \, .
}
Here $c = c(D, x_0, r, R)$ is a constant that depends only on the local geometric properties of $D$ near $x_0$, and $B(x_0, r)$ denotes the ball of radius $r$, centred at $x_0$. The boundary Harnack inequality was first proved independently by A.~Ancona (\cite{bib:a78}), B.~Dahlberg (\cite{bib:d77}) and J.-M.~Wu (\cite{bib:w78}) for Lipschitz domains, and then extended by numerous authors to a wider class of domains and elliptic operators. We refer to~\cite{bib:a01,bib:a09,bib:a14,bib:a15,bib:ls14} for further discussion and references.


Under appropriate assumptions on the regularity of $D$, the estimate~\eqref{eq:ro} turns out to be self-improving as $r \to 0^+$, in the sense that the constant $c$ in~\eqref{eq:ro} converges to $1$ as $r \to 0^+$. Equivalently, the boundary limit
\formula[eq:limit]{
 \lim_{\substack{x \to x_0 \\ x \in D}} \frac{f(x)}{g(x)}
}
exists. When $D$ is a Lipschitz domain, then in fact $c(D, x_0, r, R)$ is of order $r^\beta$ as $r \to 0^+$ for some $\beta > 0$, which means that $f / g$ extends to a H\"older continuous function at $x_0$.

A closely related concept of \emph{Martin representation} of positive harmonic functions was introduced by R.~S.~Martin in his beautiful article~\cite{bib:m41}, more than two decades before the boundary Harnack inequality became available. Given the existence of limits~\eqref{eq:limit} (for example, if $D$ is a Lipschitz domain), Martin's result asserts that there is a one-to-one correspondence between positive harmonic functions $f$ in $D$ and positive measures $\mu$ on the boundary of $D$. The two objects are linked by the formula
\formula{
 f(x) & = \int_{\partial D} M_D(x, z) \mu(dz) ,
}
where the \emph{Martin kernel} is defined as the boundary limit of the ratio of \emph{Green functions}:
\formula[eq:martin]{
 M_D(x, z) & = \lim_{\substack{y \to z \\ x \in D}} \frac{G_D(x, y)}{G_D(\tilde{x}, y)} \, .
}
Here $\tilde{x} \in D$ is an arbitrarily fixed reference point.

One of numerous equivalent definitions of harmonicity links harmonic functions with the Brownian motion: $f$ is harmonic in $D$ if and only if $f$ has the \emph{mean-value property} with respect to the distributions of the Brownian motion $X_t$ at first exit times:
\formula[eq:harm]{
 f(x) & = \ex_x f(X({\tau_U})) 
}
for all bounded open sets $U$ such that the closure of $U$ is contained in $D$. Here $\ex_x$ denotes the expectation (and $\pr_x$ will denote the probability) corresponding to the Brownian motion process $X_t$ that starts at $x$, and $\tau_U$ is the time of first exit from $U$:
\formula{
 \tau_U = \inf \{ t \ge 0 : X_t \notin U \} .
}
This probabilistic definition has a number of advantages: it extends immediately to general Markov processes $X_t$, and it captures easily boundary conditions imposed on harmonic functions. More precisely, in the general statement of the boundary Harnack inequality one requires that positive harmonic functions $f$ and $g$ converge to zero at each boundary point in $\partial D \cap B(x_0, R)$ that is \emph{regular for the Dirichlet problem}. This condition translates to requiring that~\eqref{eq:harm} holds for all bounded open sets $U$ such that $\overline{U} \sub D \cup (\partial D \cap B(x_0, R))$, with no reference to the notion of regular boundary points. Here we understand that $f = g = 0$ in $\partial D \cap B(x_0, R)$.

In this article we are interested in Markov processes with jumps, and from now on by saying that a function is harmonic we understand that it has the mean-value property~\eqref{eq:harm} with respect to a Markov process $X_t$ with jumps. In this case in order to evaluate $f(X(\tau_U))$ in~\eqref{eq:harm} the function $f$ needs to be defined everywhere, not just in $D$. For this reason one needs to replace the \emph{boundary} condition $f = g = 0$ in $\partial D \cap B(x_0, R)$ in the statement of the boundary Harnack inequality with the \emph{exterior} condition $f = g = 0$ in $D^c \cap B(x_0, R)$
.

The history of the boundary Harnack inequality for Markov processes with jumps starts with the article by K.~Bogdan (\cite{bib:b97}), where he proved the result for the isotropic stable L\'evy process (equivalently: for the fractional Laplace operator $-(-\Delta)^{\alpha/2}$) and Lipschitz domains. Later this was extended to more general sets (\cite{bib:sw99,bib:bkk08}) and processes (\cite{bib:bbc03,bib:cksv12,bib:g07,bib:ksv09,bib:ksv12,bib:ksv13,bib:ksv14a,bib:ksv14b}). Recently, a rather general result for Markov processes with jumps was proved in~\cite{bib:bkk15}, and this is our starting point in the study of boundary limits~\eqref{eq:limit}.

The existence of the boundary limit~\eqref{eq:limit} in this context was first proved independently by K.~Bogdan (\cite{bib:b99}) and by Z.-Q.~Chen and R.~Song (\cite{bib:cs98}) for the isotropic stable L\'evy process and Lipschitz domains. This required an appropriate modification of the classical reasoning due to the presence of jumps. Since then essentially every time the boundary Harnack inequality was established for a given Markov process with jumps in a given class of domains, the existence of boundary limits~\eqref{eq:limit} followed; see~\cite{bib:ksv15} for the most recent result of this kind. With one exception, however, the class of open sets under consideration was always limited to certain disconnected analogues of non-tangentially accessible domains, typically called \emph{fat sets}. The single more general result is proved in~\cite{bib:bkk08} for the isotropic stable L\'evy process, where completely arbitrary open sets are allowed.

For the existence of boundary limits, we follow the approach of~\cite{bib:bkk08} using the boundary Harnack inequality of~\cite{bib:bkk15}, and prove in our main results, Theorems~\ref{thm:martin} and~\ref{thm:repr}, the existence of boundary limits of ratios of harmonic functions for arbitrary open sets and rather general Markov processes with jumps, as well as Martin representation of such functions. The application of the method developed in~\cite{bib:bkk08} in the present setting requires significant modifications. Further changes are introduced in order to make the description of the proof more accessible; for example, we first give a simpler argument which does not assert uniform convergence with respect to the domain of harmonicity, and only then explain how one improves it to get a domain-uniform version.

The proof of Martin representation theorem for the isotropic stable L\'evy processes in~\cite{bib:bkk08} is self-contained. It is possible to extend the method of~\cite{bib:bkk08} to our general setting, but that would require rather lengthy and technical arguments. For this reason, unlike in~\cite{bib:bkk08}, we refer to the general theory of Martin boundary. Our argument still requires extension of some elements of~\cite{bib:bkk08} for more general Markov processes, but the most involved part of the proof is avoided. For an excellent exposition of the general theory of Martin boundary, we refer to Chapter~14 of~\cite{bib:cw05}.

We conclude the introduction with a description of the structure of this article. The assumptions for the boundary Harnack inequality of~\cite{bib:bkk15} are briefly recalled in Section~\ref{sec:bhi}. We omit a detailed discussion of these conditions and refer the interested reader to the original paper. Instead, we present a number of examples right after the statement of Theorems~\ref{thm:martin} and~\ref{thm:repr} in Section~\ref{sec:main}. We also provide a counter-example, which shows that the boundary limits~\eqref{eq:limit} typically fail to exist in irregular domains when the process $X_t$ has a non-trivial diffusion part. Finally, in Section~\ref{sec:proof} we prove Theorems~\ref{thm:martin} and~\ref{thm:repr}.

%
%

\section{Fundamental assumptions for the boundary Harnack inequality}
\label{sec:bhi}

The formal statement of the assumptions for Theorem~\ref{thm:martin} requires some effort. We assume that $(\X, d, m)$ is a locally compact metric measure space in which all bounded closed sets are compact and $m$ has full support, and that $R_0 > 0$ (possibly $R_0 = \infty$) is a localisation radius such that $\X \setminus B(x, r) \ne \varnothing$ if $x \in \X$ and $0 < r < 2 R_0$.

In~\cite{bib:bkk15} the following four conditions are introduced. A detailed discussion of these assumptions is beyond the scope of the present article, we refer the reader to~\cite{bib:bkk15} for more information. Here we only state the conditions, without explaining in a formal way the notions of \emph{semi-polar} and \emph{polar} sets, \emph{processes in duality} $X_t$ and $\hat{X}_t$, their \emph{generators} $\A$ and $\hat{\A}$, densities $\nu(x, y)$ and $\hat{\nu}(x, y)$ (with respect to the measure $m$) of the \emph{L\'evy kernels} of $X_t$ and $\hat{X}_t$, as well as their \emph{Green functions} $G_D(x, y) = \hat{G}_D(y, x)$. We note that $\nu(x, y)$ describes the intensity of jumps from $x$ to $y$ and it is commonly used throughout the article. The Green function $G_D(x, y)$ is required for Theorem~\ref{thm:repr} only; informally, $G_D(x, y)$ is the average amount of time spent near $y$ by the process $X_t$, started at $x$, until $\tau_D$.

\begin{assumption}
\label{asm:dual}
Hunt processes $X_t$ and $\hat{X}_t$ are dual with respect to the measure $m$. The transition semigroups of $X_t$ and $\hat{X}_t$ are both Feller and strong Feller. Every semi-polar set of $X_t$ is polar.
\end{assumption}

\begin{assumption}
\label{asm:gen}
There is a linear subspace $\dom$ of $\dom(\A) \cap \dom(\hat{\A})$ satisfying the following condition. If $K$ is compact, $D$ is open, and $K \sub D \sub \X$, then there is $f \in \dom$ such that $f(x) = 1$ for $x \in K$, $f(x) = 0$ for $x \in \X \setminus D$, $0 \le f(x) \le 1$ for $x \in \X$, and the boundary of the set $\{x \, : \, f(x) > 0\}$ has measure $m$ zero.
\end{assumption}

\begin{assumption}
\label{asm:levy}
We have $\nu(x, y) = \hat{\nu}(y, x) > 0$ for all $x, y \in \X$, $x \ne y$. If $x_0 \in \X$, $0 < r < R < R_0$, $x \in B(x_0, r)$ and $y \in \X \setminus B(x_0, R)$, then
\formula[eq:nu]{
 \cnu^{-1}{\nu(x_0, y)} \le {\nu(x, y)} \le \cnu{\nu(x_0, y)} , &&
 \cnu^{-1}{\hat\nu(x_0, y)} \le {\hat\nu(x, y)} \le \cnu{\hat\nu(x_0, y)} , 
}
with $\cnu = \cnu(x_0, r, R)$. 
\end{assumption}

\begin{assumption}
\label{asm:green}
If $x_0 \in \X$, $0 < r < s < R < R_0$ and $B = B(x_0, R)$, then 
\begin{equation}
\label{eq:cgreen}
 \cgreen = \cgreen(x_0, r, s, R) = \sup_{x \in B(x_0, r)} \sup_{y \in \X \setminus B(x_0, s)} \max(G_B(x, y), \hat{G}_B(x, y)) < \infty .
\end{equation}
\end{assumption}

We denote
\formula[eq:cro]{
 \cro(K, D) & = \inf_{f} \sup_{x \in \X} \max(\A f(x), \hat{\A} f(x)) ,
}
where the infimum is taken over all functions $f$ described by the Assumption~\ref{asm:gen}. If $x_0 \in \X$ and $0 < r < R < R_0$, then we denote
\formula{
 \cnuinf(x_0, r, R) & = \inf_{y \in \overline{B}(x_0, R) \setminus B(x_0, r)} \min(\nu(x_0, y), \hat{\nu}(x_0, y)) ,
}
and
\formula{
 \ctau(x_0, r) & = \sup_{x \in B(x_0, r)} \max(\ex_x \tau_{B(x_0, r)}, \hat{\ex}_x \hat{\tau}_{B(x_0, r)}) .
}
Following~\cite{bib:b97}, we say that $f$ is a \emph{regular harmonic function} in an open set $D$ if the mean-value property~\eqref{eq:harm} holds with $U = D$. By the strong Markov property, this implies that~\eqref{eq:harm} holds for arbitrary open $U \sub D$, so in particular $f$ is harmonic in $D$. The following theorem is a reformulation of the main result of~\cite{bib:bkk15}.

\begin{theorem}[{Lemma~3.2 and Theorems~3.4 and~3.5 in~\cite{bib:bkk15}}]
\label{thm:bhi}
Suppose that $x_0 \in \X$, $0 < r_1 < r_2 < r_3 < r_6 < R_0$ and a non-negative function $f$ is a regular harmonic function in $D \cap B(x_0, r_6)$, which is equal to zero in $B(x_0, r_6) \setminus D$. Then
\formula{
 f(x) & \approx \cbhi \ex_x \tau_{D \cap B(x_0, r_2)} \int_{\X \setminus B(x_0, r_3)} f(y) \nu(x_0, y) m(dy)
}
for $x \in D \cap B(x_0, r_1)$, where $\cbhi = \cbhi(x_0, r_1, r_2, r_3, r_6)$ is defined as
\formula{
 \cbhi = \cnu(x_0, r_2, r_3) & + 2 \cro(\overline{B}(x_0, r_3) \setminus B(x_0, r_2), B(x_0, r_8) \setminus \overline{B}(x_0, r_1)) \times \\
 & \times \expr{\cgreen(x_0, r_3, r_4, r_6) + \frac{\ctau(x_0, r_6) (\cnu(x_0, r_4, r_5))^2}{m(B(x_0, r_4))}} \times \\
 & \times \expr{\frac{\cro(\overline{B}(x_0, r_5), B(x_0, r_6))}{\cnuinf(x_0, r_5, r_7)} + \cnu(x_0, r_6, r_7) m(B(x_0, r_6))}
}
for some $r_4$, $r_5$, $r_7, r_8$ such that $0 < r_1 < r_2 < r_3 < r_4 < r_5 < r_6 < r_7 < r_8$.
\end{theorem}

Note that it is important that $f$ is non-negative everywhere, not just in $D$. Theorem~\ref{thm:bhi} implies the more classical statement of the boundary Harnack inequality (Theorem~3.5 in~\cite{bib:bkk15}): if $f$ and $g$ satisfy the assumptions of Theorem~\ref{thm:bhi}, then
\formula[eq:bhi]{
 \sup_{x \in D_r} \frac{f(x)}{g(x)} & \le \cbhi^4 \inf_{x \in D_r} \frac{f(x)}{g(x)} \, ,
}
as in~\eqref{eq:ro}. We remark that although the original statement allows for an arbitrary sequence of radii, it will be sufficient for us to consider $r_1 = r$, $r_2 = 2 r$, $r_3 = 3 r$ and $r_6 = 4 r$, and we will commonly write $\cbhi = \cbhi(x_0, r) = \cbhi(x_0, r, 2r, 3r, 4r)$ in this case.

%
%

\section{Main results and examples}
\label{sec:main}

For the existence of limits, we introduce one more definition. If $x_0 \in \X$ and $0 < r < R < R_0$, we let
\formula[eq:nuint]{
 \cnuint = \cnuint(x_0, r, R) = \frac{\int_{\X \setminus B(x_0, r)} \nu(x_0, y) m(dy)}{\int_{\X \setminus B(x_0, R)} \nu(x_0, y) m(dy)} \, .
}
We use a short-hand notation $f \approx c g$ for the two inequalities $c^{-1} g \le f \le c g$, where $c > 0$ is a positive constant.

\begin{theorem}
\label{thm:martin}
Let $D \sub \X$ be open, $x_0 \in \partial D$ and $R > 0$. Suppose that:
\begin{enumerate}[label={\rm(\roman*)}]
\item\label{it:martin:1} $X_t$ satisfies Assumptions~\ref{asm:dual} through~\ref{asm:green};
\item\label{it:martin:2} $\lim\limits_{r \to 0^+} \cnu(x_0, r, R) = 1$;
\item\label{it:martin:3} the constant $\cnu(x_0, r, 2r)$ is bounded in $r$, $0 < 2r < R_0$;
\item\label{it:martin:4} the constant $\cnuint(x_0, r, 2r)$ is bounded in $r$, $0 < 2r < R_0$;
\item\label{it:martin:5} the constant $\cbhi(x_0, r, 2r, 3r, 4r)$ is bounded in $r$, $0 < 4r < R_0$.
\end{enumerate}
Suppose furthermore that non-negative functions $f$ and $g$ are regular harmonic functions in $D \cap B(x_0, R)$ and are equal to zero in $B(x_0, R) \setminus D$. Then either one of $f$ and $g$ is zero everywhere in $D$, or the finite, positive boundary limit of $f(x) / g(x)$ exists as $x \to x_0$, $x \in D$. Furthermore,
\formula[eq:thm:martin]{
 \lim_{\substack{x \to x_0 \\ x \in D}} \frac{f(x)}{g(x)} & = \lim_{r \to 0^+} \frac{\int_{\X \setminus B(0, r)} \nu(x_0, y) f(y) m(dy)}{\int_{\X \setminus B(0, r)} \nu(x_0, y) g(y) m(dy)} \, .
}
\end{theorem}

\begin{remark}
\label{rem:accessible}
Condition~\ref{it:martin:2} is required only for \emph{inaccessible} boundary points $x_0$, characterised by the property $\int_{D \cap B(x_0, R)} \ex_y \tau_{D \cap B(x_0, R)} m(dy) < \infty$. The result for \emph{accessible} boundary points $x_0$, for which the integral is infinite, holds under conditions~\ref{it:martin:1} and~\ref{it:martin:3} through~\ref{it:martin:5}.
\end{remark}

\begin{remark}
\label{rem:exit}
Theorem~\ref{thm:martin} also holds with $g(x) = \ex^x \tau_{D \cap B(x_0, R)}$. This is formally shown in Section~\ref{subsec:ex}, but the informal explanation is rather straightforward: $g$ is \emph{essentially} a regular harmonic function in $D \cap B(x_0, R)$ (in sharp contrast with the case of continuous Markov processes).

Indeed, suppose that $\X$ is unbounded, $D$ is a bounded open set and that $\cnu(x_0, r, R)$ converges to $1$ as $R \to \infty$. By Dynkin's formula (see Lemma~\ref{lem:dynkin} and estimate~\eqref{eq:dynkin} below),
\formula{
 \ex^x \tau_D & = \lim_{R \to \infty} \frac{\pr_x(X({\tau_D}) \in \X \setminus B(x_0, R))}{\int_{\X \setminus B(x_0, R)} \nu(x_0, y) m(dy)}
}
is the limit of regular harmonic functions in $D$. Since the estimates in Theorem~\ref{thm:martin} are uniform in $f$ and $g$, we obtain the desired result. (Note that the formal argument is completely different and requires no further assumptions on $\X$ and $X_t$.)
\end{remark}

\begin{remark}
\label{rem:uniform}
As remarked in the introduction, the limit in~\eqref{eq:thm:martin} exists if and only if the relative oscillation of $f$ and $g$ converges to one, that is,
\formula{
 \lim_{r \to 0^+} \frac{\sup_{x \in D \cap B(x_0, r)} (f(x)/g(x))}{\inf_{x \in D \cap B(x_0, r)} (f(x)/g(x))} & = 1 .
}
By inspecting the proof of Theorem~\ref{thm:martin}, one immediately sees that, given $D$ and $x_0$, the boundary limits exist uniformly in $f$ and $g$, in the sense that
\formula{
 \lim_{r \to 0^+} \sup_{f,g} \frac{\sup_{x \in D \cap B(x_0, r)} (f(x)/g(x))}{\inf_{x \in D \cap B(x_0, r)} (f(x)/g(x))} & = 1 ,
}
with the supremum taken over all $f$ and $g$ satisfying the assumptions of the theorem. We remark that in fact one can prove uniformity also in $D$, just as in~\cite{bib:bkk08}, by appropriately modifying the final part of the proof. More formally,
\formula[eq:uniform]{
 \lim_{r \to 0^+} \sup_{D,f,g} \frac{\sup_{x \in D \cap B(x_0, r)} (f(x)/g(x))}{\inf_{x \in D \cap B(x_0, r)} (f(x)/g(x))} & = 1 ,
}
where the supremum is taken over all open sets $D$ and $f$ and $g$ satisfying the assumptions of the theorem (here we let the ratio $\sup/\inf$ be equal to $1$ if $D \cap B(x_0, r)$ is empty). The proof of this result is sketched in Section~\ref{subsec:ex}.
\end{remark}

\begin{remark}
\label{rem:continuity}
It is not necessary to assume that $x_0 \in \partial D$ in Theorem~\ref{thm:martin}. For $x_0 \notin D$ the statement is void, but for $x_0 \in D$ we obtain \emph{relative continuity} of positive harmonic functions: if $f$ and $g$ are positive harmonic functions in $D$, then $f / g$ is continuous in $D$. If the process is conservative, then the constant $1$ is harmonic, and consequently positive harmonic functions are continuous. If in addition the characteristics of the process (that is, the constants in conditions~\ref{it:martin:2} through~\ref{it:martin:5}) do not depend on $x_0$, then harmonic functions are in fact uniformly continuous (see also Remark~\ref{rem:uniform}).
\end{remark}

Before we discuss examples, we provide one application. Recall that the Green function $G_D(x, y)$ is the density of the mean occupation measure of $X_t$ up to $\tau_D$, that is,
\formula{
 \int_A G_D(x, y) m(dy) & = \ex_x \int_0^{\tau_D} \ind_A(X_s) ds .
}
Under Assumptions~\ref{asm:dual} and~\ref{asm:green}, there is a version of $G_D(x, y)$ which is a harmonic function of $x \in D \setminus \{y\}$, and a co-harmonic (that is, harmonic for the dual process) function of $y \in D \setminus \{x\}$. Hence, Theorem~\ref{thm:martin} (or, more precisely, its version for the dual process) immediately implies the existence of the Martin kernel 
\formula{
 M_D(x, z) & = \lim_{\substack{y \to z \\ x \in D}} \frac{G_D(x, y)}{G_D(\tilde{x}, y)} \, .
}
for $z = x_0$ (this is exactly the same as the classical definition~\eqref{eq:martin}). Informally, the Martin boundary $\partial_M D$ of a set $D$ is the set of all possible ways a point $y \in D$ approaches the boundary in such a way that the ratio $G_D(x, y) / G_D(\tilde{x}, y)$ converges for every $x \in D$ (with arbitrarily fixed $\tilde{x} \in D$). More formally, $D \cup \partial_M D$ is the \emph{Constantinescu--Cornea compactification} of $D$ with respect to the family of functions \mbox{$\{ G_D(x, \cdot) / G_D(\tilde{x}, \cdot) : x \in D \}$}: the smallest compact space which contains $D$ and on which these functions have continuous extensions.

\begin{theorem}
\label{thm:repr}
Let $D \sub \X$ be bounded and open, and if $\X$ is compact, then assume in addition that $\ex_x(\tau_D)$ is finite and bounded in $x \in D$. Suppose that the assumptions of Theorem~\ref{thm:martin} are satisfied \emph{uniformly} for all $x_0 \in \overline{D}$. Then the following assertions hold.
\begin{enumerate}[label={\rm(\alph*)}]
\item\label{it:repr:1} The Martin boundary $\partial_M D$ coincides with the topological boundary $\partial D$.
\item\label{it:repr:2} The Martin kernel $M_D(x, z)$ is a harmonic function in $D$ with respect to $x$ if and only if $z$ is an accessible boundary point: \mbox{$\int_{D \cap B(x_0, R)} \ex_y \tau_{D \cap B(x_0, R)} m(dy) = \infty$}.
\item\label{it:repr:3} In this case $M_D(x, z)$ is a \emph{minimal} harmonic function: if $f$ is a harmonic function in $D$ and $0 \le f(x) \le M_D(x, z)$ for all $x \in \X$, then $f(x)$ is a multiple of $M_D(x, z)$.
\item\label{it:repr:4} Every non-negative function $f$ which is a harmonic function in $D$ has a unique representation
\formula[eq:repr]{
\begin{aligned}
 f(x) & = \int_{\X \setminus (D \cup \partial_m D)} \expr{\int_D G_D(x, y) \nu(y, z) m(dy)} f(z) m(dz) \\
 &  \hspace*{15em} + \int_{\partial_m D} M_D(x, z) \mu(dz) ,
\end{aligned}
}
where $\mu$ is a measure on $\partial_m D$, the set of \emph{accessible} boundary points of $D$.
\item\label{it:repr:5} Conversely, given any non-negative function $f$ and any measure $\mu$ on $\partial_m D$, the right-hand side of~\eqref{eq:repr} is either a harmonic function in $D$ or infinity everywhere in $D$.
\end{enumerate}
\end{theorem}

\begin{remark}
The terms \emph{accessible} and \emph{inaccessible} correspond to the probabilistic theory of Martin boundary. To be specific, the process $X_t$ killed at the time of first exit from $D$ and conditioned in the sense of Doob by the Martin kernel $M_D(\cdot, z)$ converges at its lifetime to $z$ when $z$ is accessible, and dies out in $D$ when $z$ is inaccessible. We refer to~\cite{bib:cw05} for more information.
\end{remark}

\begin{remark}
Unlike in the case of isotropic stable L\'evy processes in~\cite{bib:bkk08}, description of the infinite part of the Martin boundary of $D$ for unbounded open sets is a completely different problem.
\end{remark}

The boundary Harnack inequality stated in Theorem~\ref{thm:bhi} was applied to a variety of Markov processes in Section~5 of~\cite{bib:bkk15}. The \emph{scale-invariant} version of Theorem~\ref{thm:bhi} under \emph{$\alpha$-stable-like scaling} discussed therein already asserts conditions~\ref{it:martin:1}, \ref{it:martin:3} and~\ref{it:martin:5} in Theorem~\ref{thm:martin}. Verification of the remaining conditions~\ref{it:martin:2} and~\ref{it:martin:4} is typically straightforward, and we obtain several classes of processes for which Theorems~\ref{thm:martin} and~\ref{thm:repr} apply.

In our the first example, we use the result of Example~5.5 in~\cite{bib:bkk15}, where boundary Harnack inequality for L\'evy processes is considered. In the asymmetric case, equality of the notions of semi-polar and polar sets (in Assumption~\ref{asm:dual}) is not trivial,  and this was apparently overlooked in~\cite{bib:bkk15}. Fortunately, for all asymmetric L\'evy processes listed therein, this condition is satisfied by Theorem~2 in~\cite{bib:r88}.

\begin{example}[Strictly stable L\'evy processes]
Let $m$ be the Lebesgue measure in $\R^d$, $R_0 = \infty$. Suppose that $X_t$ is a strictly $\alpha$-stable L\'evy process in $\R^d$, where $d \ge 1$ and $0 < \alpha < 2$. Suppose, furthermore, that the L\'evy measure of $X_t$ has a density function of the form $\nu(z) = \ph(z / |z|) |z|^{-d - \alpha}$, with $\ph$ continuous and positive on the unit sphere (for L\'evy processes, $\nu(x, y) = \nu(y - x)$). It is easy to see that $\cnu(x_0, r, R)$ converges to $1$ as $r \to 0^+$ and that $\cnuint(x_0, r, R) = (R / r)^\alpha$. By Example~5.5 in~\cite{bib:bkk15}, $X_t$ satisfies the other assumptions of Theorem~\ref{thm:martin}, and so we may use Theorems~\ref{thm:martin} and~\ref{thm:repr}.
\end{example}

We remark that the above example can be extended to more general L\'evy processes, including many subordinate Brownian motions and, more generally, unimodal isotropic L\'evy processes. This is based on estimates obtained recently in~\cite{bib:bgr14,bib:bgr15,bib:g14,bib:gr15} and will be studied in detail in~\cite{bib:gk15}. Another extensions can be obtained by allowing the L\'evy kernel to depend on $x$ or restricting it to a domain, as described in the following two examples.

\begin{example}[Stable-like processes]
Let $m$ be the Lebesgue measure in $\R^d$, $R_0 = \infty$. Suppose that $0 < \alpha < 2$ and
\formula{
 \nu(x, y) & = \ph(x, y) |x - y|^{-d - \alpha} ,
}
where $\ph$ is symmetric (that is, $\ph(x, y) = \ph(y, x)$), bounded by positive constants, smooth, and has bounded partial derivatives of all orders. As in Example~5.6 in~\cite{bib:bkk15}, in this case there is a pure-jump process $X_t$ with the L\'evy kernel $\nu(x, y) m(dy)$, and the assumptions of Theorem~\ref{thm:martin} are satisfied.
\end{example}

\begin{example}[Reflected stable processes]
Let $0 < \alpha < 2$. Let $\X$ be the closure of either a Lipschitz domain in $\R^d$ if $\alpha < 1$ or a $C^{1,\alpha+\eps}$ domain in $\R^d$ if $\alpha \ge 1$ (with some $\eps > 0$). Let $m$ be the Lebesgue measure on $\X$, and $\nu(x, y) = c |x - y|^{-d - \alpha}$ for some $c > 0$. Again as in Example~5.6 in~\cite{bib:bkk15}, there is a pure-jump process $X_t$ with the L\'evy kernel $\nu(x, y) m(dy)$, and the assumptions of Theorem~\ref{thm:martin} are satisfied for some $R_0 > 0$.
\end{example}

The state space $\X$ need not be Euclidean.

\begin{example}[Stable-like subordinate diffusions]
Let $\X$ be a sufficiently regular metric measure space in which there exists a diffusion process. For a rigorous definition, we refer to Example~5.7 in~\cite{bib:bkk15}; examples include Riemannian manifolds, Sierpi\'nski gaskets or the Sierpi\'nski carpet. Suppose that $0 < \alpha < \dwalk$, where $\dwalk$ is the \emph{walk dimension} of $\X$ (that is, an approximate scaling exponent for the diffusion process). Finally, let $X_t$ be a process subordinate to the diffusion process, corresponding to the $(\alpha/\dwalk)$-stable subordinator. In Example~5.7 in~\cite{bib:bkk15} it is shown that $X_t$ satisfies conditions~\ref{it:martin:1}, \ref{it:martin:3} and~\ref{it:martin:5} of Theorem~\ref{thm:martin}, and one easily proves that $\cnuint(x_0, r, R) \le c (R / r)^\alpha$ for some $c > 0$. Verification of~\ref{it:martin:2} requires some work, especially when $\X$ is unbounded. For this reason, we only sketch the argument for compact $\X$. For some $c > 0$ we have
\formula{
 \nu(x, y) & = c \int_0^\infty t^{-1 - \alpha/\dwalk} q_t(x, y) dt ,
}
where $q_t(x, y)$ is the transition density of the diffusion process. Since for each $t > 0$, $q_t$ is H\"older continuous, it is easy to see that $\nu(x, y)$ is positive and uniformly continuous in $x \in \overline{B}(x_0, r)$, $y \in \X \setminus B(x_0, R)$, which clearly implies condition~\ref{it:martin:2}. It follows that Theorems~\ref{thm:martin} and~\ref{thm:repr} apply to stable-like subordinate diffusions in compact metric measure spaces.
\end{example}

Surprisingly, Theorem~\ref{thm:martin} is not influenced by killing.

\begin{example}[Processes with a multiplicative functional]
Let $M_t$ be a strong continuous multiplicative functional such that $M_0 = 1$ with probability one for all starting points $x \in \X$. Such a functional describes gradual killing of the process $X_t$, and is typically obtained as the Feynman--Kac functional $M_t = \exp(-\int_0^t V(X_s) ds)$ for some non-negative function $V$. A function $f$ is said to be harmonic with respect to the pair $(X_t, M_t)$ if it has the mean-value property
\formula{
 f(x) & = \ex_x (f(X({\tau_U})) M({\tau_U}))
}
instead of~\eqref{eq:harm}. As in Theorem~5.10 in~\cite{bib:bkk15}, if the assumptions of Theorem~\ref{thm:martin} are satisfied by the process $X_t$, then the conclusion also holds for functions harmonic with respect to the pair $(X_t, M_t)$.
\end{example}

Our final example shows that when $X_t$ has non-vanishing diffusion part, one cannot expect the existence of boundary limits~\eqref{eq:limit} unless some geometric restrictions on $D$ are imposed. For corresponding positive results in smooth domains, see~\cite{bib:ksv13}.

\begin{example}[Mixture of Brownian motion and stable process]
Let $\X = \R$ and let $m$ be the Lebesgue measure. Let $X_t$ be a one-dimensional L\'evy process which is the sum of two independent L\'evy processes: the Brownian motion and the symmetric $\alpha$-stable Lévy process for some $\alpha \in (1, 2)$. That is, the characteristic exponent of $X_t$ is given by $c_1 \xi^2 + c_2 |\xi|^\alpha$ for some $c_1, c_2 > 0$. Denote $D = (-1, 1) \setminus \{0\}$. Let $p_t(y - x)$ be the continuous version of the transition density of $X_t$. Then the three functions
\formula{
 u(x) & = x, & v(x) & = \int_0^\infty (p_t(0) - p_t(x)) dt , & w(x) & = \ex_x |X({\tau_D})|
}
are regular harmonic in $D$: for $u$ this is just the martingale property of $X_t$, for $v$ (the \emph{compensated potential kernel} of $X_t$) this is proved, for example, in~\cite{bib:y10}, while for $w$ it follows directly from the definition. Furthermore, $u(0) = v(0) = w(0) = 0$ and $v(x) = v(-x)$, $w(x) = w(-x)$. It is known that
\formula{
 v(x) & \approx c_3 \min(|x|, |x|^{\alpha-1})
}
for $x \in \R$, with $c_3 = c_3(c_1, c_2, \alpha)$ (see, for example, Lemma~2.14 in~\cite{bib:gr15}). In particular, $v(x) \approx c_3 |x|$ for $x \in D$. Finally, by the boundary Harnack inequality given in Theorem~\ref{thm:bhi} (see Examples~5.5 and~5.13 in~\cite{bib:bkk15} for a detailed discussion), we have
\formula{
 w(x) & \approx c_4 v(x) \approx c_3 c_4 |x|
}
for $x \in (-\tfrac{1}{2}, \tfrac{1}{2})$, with $c_4 = c_4(c_1, c_2, \alpha)$. Let us define
\formula{
 f(x) & = w(x) + u(x) = 2 \ex_x (|X({\tau_D})| \ind_{[1, \infty)}(X({\tau_D}))) , \\
 g(x) & = w(x) - u(x) = 2 \ex_x (|X({\tau_D})| \ind_{(-\infty, -1]}(X({\tau_D}))) .
}
Then $f$ and $g$ are non-negative, regular harmonic in $D$ and equal to zero in $(-1, 1) \setminus D = \{0\}$, so that they satisfy the assumptions of Theorem~\ref{thm:martin}. On the other hand,
\formula{
 \frac{f(x)}{g(x)} - \frac{f(-x)}{g(-x)} & = \frac{w(x) + x}{w(x) - x} - \frac{w(x) - x}{w(x) + x} = \frac{4 x w(x)}{(w(x))^2 - x^2}
}
for $x \in D$. Since $t / (t^2 - x^2)$ is decreasing in $t \in (x, \infty)$, and $w(x) \le c_3 c_4 x$ for $x \in (0, \tfrac{1}{2})$, we obtain
\formula{
 \frac{f(x)}{g(x)} - \frac{f(-x)}{g(-x)} & \ge \frac{4 c_3 c_4}{(c_3 c_4)^2 - 1}
}
for $x \in (0, \tfrac{1}{2})$. In particular, the limit of $f(x) / g(x)$ as $x \to 0$ does not exist. 
\end{example}

%
%

\section{Proofs of main results}
\label{sec:proof}

In this section we prove Theorem~\ref{thm:martin}. We will always assume that $x_0$, $R$ and $D$ are fixed, where $x_0 \in \X$, $0 < 2 R < R_0$ and $D \sub B(x_0, R)$ is an open set. It is also understood that $x_0 \in \partial D$, although, at least formally, the argument extends also to $x \in D$ and $x \notin \overline{D}$. Recall that the notation $f \approx c g$ stands for $c^{-1} g \le f \le c g$ with $c > 0$.

We denote $B_r = B(x_0, r)$, $B_{r,s} = B_s \setminus B_r$, $D_r = D \cap B_r$ and $D_{r,s} = D_s \setminus D_r$ when $0 \le r \le s \le R$. We furthermore define $D_{r,\infty} = D_{r,R} \cup (\X \setminus B_R)$. For a non-negative function $f$ we let
\formula{
 M_{r,\infty}(f) & = \int_{\X \setminus B_r} f(y) \nu(x_0, y) m(dy) , & M_{r,s}(f) & = \int_{B_{r,s}} f(y) \nu(x_0, y) m(dy) .
}
Finally, we let $s_D(x) = \ex_x \tau_D$.

To simplify the notation, we drop $D$ from the notation in subscripts whenever possible, and we write $\tau_r = \tau_{D_r}$, $\tau_{r,s} = \tau_{D_{r,s}}$, $s_r(x) = s_{D_r}(x)$, $\ind_{r,s}(x) = \ind_{D_{r,s}}(x)$ etc.

Our argument is based on the boundary Harnack inequality of~\cite{bib:bkk15}, stated in Theorem~\ref{thm:bhi}. Under the assumptions of Theorem~\ref{thm:martin}, the constant $\cbhi(x_0, r, 2r, 3r, 4r)$ can be chosen so that it does not depend on $r$, as long as $0 < 4 r \le R$, and it will be denoted simply by $\cbhi$ (recall that $x_0$ and $R$ are fixed). In a similar way, we denote $\cnu = \cnu(x_0, r, 2r)$ (with $0 < 2 r < R_0$) and $\cnuint = \cnuint(x_0, r, 2 r)$ (with $0 < 2 r < R_0$), chosen independently of $r$. With one exception, we will only use constants $\cbhi$, $\cnu$ and $\cnuint$ with these parameters.

We prove Theorem~\ref{thm:martin} by considering separately two types of boundary points, which are called \emph{accessible} and \emph{inaccessible} in~\cite{bib:bkk08}. First, however, we introduce some further notation and prove preliminary estimates.

\subsection{Decomposition of harmonic functions}

From now on $f$ and $g$ are functions satisfying the assumptions of Theorem~\ref{thm:martin}, and we assume that neither $f$ nor $g$ is equal to zero almost everywhere. Note that this implies that $f$ and $g$ are strictly positive in $D$. Whenever $0 < r < s \le R$, we decompose $f$ into the sum of two functions, $f_{r,s}$ and $\tilde{f}_{r,s}$, which correspond to the process $X_t$ exiting $D_r$ near its boundary (into $D_{r,s}$) and away of its boundary (into $D_{s,\infty}$):
\formula{
 f_{r,s}(x) & = \ex_x((f \ind_{{r,s}})(X(\tau_{r}))) , &
 \tilde{f}_{r,s}(x) & = \ex_x((f \ind_{{s,\infty}})(X(\tau_{r}))) .
}
Not unexpectedly, a similar notation is used for the function $g$. Clearly, $f = f_{r,s} + \tilde{f}_{r,s}$, and both $f_{r,s}$ and $\tilde{f}_{r,s}$ are non-negative regular harmonic harmonic in $D_r$ which are equal to zero in $B_r \setminus D_r$. Therefore, we can apply Theorem~\ref{thm:bhi} to $f_{4r,s}$ and $\tilde{f}_{4r,s}$ whenever $0 < 4 r < s \le R$.

Note that by Theorem~\ref{thm:bhi} (with $r = \tfrac{R}{4}$), we have
\formula{
 f(x) & \approx \cbhi M_{3R/4,\infty}(f) \ex_x \tau_{{2R/4}}
}
for $x \in D_{R/4}$. Therefore,
\formula[eq:mf:ms]{
 M_{r,s}(f) & \approx \cbhi M_{3R/4,\infty}(f) M_{r,s}(s_{{R/2}})
}
whenever $0 \le r \le s \le \tfrac{R}{4}$. The next result states, in particular, that there is little difference whether we write \smash{$s_{{R/2}}$} or \smash{$s_{R}$} in the above estimate.

\begin{lemma}
\label{lem:sdr}
If $0 < 8 r \le R$, then
\formula{
 \ex_x \tau_{{4r}} & \le \ex_x \tau_{{8r}} \le (1 + \cbhi \cnu \cnuint^3) \ex_x \tau_{{4r}}
}
for $x \in D_r$.
\end{lemma}

\begin{proof}
The first inequality is clear. For the other one, we use strong Markov property and Theorem~\ref{thm:bhi}:
\formula{
 \ex_x \tau_{{8r}} - \ex_x \tau_{{4r}} & = \ex_x s_{{8r}}(X(\tau_{{4r}})) \\
 & \le \cbhi \ex_x \tau_{{2r}} \int_{\X \setminus B_{3r}} \ex_y s_{{8r}}(X(\tau_{{4r}})) \nu(x_0, y) m(dy) \\
 & \le \cbhi \ex_x \tau_{{4r}} \int_{\X \setminus B_{2r}} \ex_y \tau_{{8r}} \nu(x_0, y) m(dy) .
}
Furthermore, by Proposition~2.1 in~\cite{bib:bkk15} (combined with the last displayed formula in the proof of this result),
\formula{
 \int_{\X \setminus B_{2r}} \ex_y \tau_{{8r}} \nu(x_0, y) m(dy) & \le \expr{\sup_{x \in \X} \ex_x \tau_{B_{8r}}} \int_{\X \setminus B_{2r}} \nu(x_0, y) m(dy) \\
 & \le \cnu \, \frac{\int_{\X \setminus B_{2r}} \nu(x_0, y) m(dy)}{\int_{\X \setminus B_{16 r}} \nu(x_0, y) m(dy)} .
}
It remains to use~\eqref{eq:nuint}.
\end{proof}

For convenience, we denote
\formula{
 \caux = 1 + \cbhi \cnu \cnuint^3 ,
}
so that $s_{4r}(x) \approx \caux s_{8r}(x)$ if $0 < 8 r \le R$ and $x \in D_r$.

Our next result compares $f_{8r,s}$ with $\tilde{f}_{8r,s}$. For $f_{8r,s}$, we will use Theorem~\ref{thm:bhi}, which states that in $D_{2r}$ we have $f_{8r,s} \approx \cbhi M_{6r,\infty}(f_{8r,s}) \ex_x \tau_{{4r}}$. The same estimate can be written down for $\tilde{f}_{8r,s}$. However, $M_{6r,\infty}(\tilde{f}_{8r,s})$ involves an integral of $\tilde{f}_{8r,s}$ over $D_{6r,8r}$, which is often problematic. A much better estimate for $\tilde{f}_{8r,s}$ can be easily obtained from the following corollary of Dynkin's formula for $X_t$.

\begin{lemma}[{formula~(2.12) in~\cite{bib:bkk15}}]
\label{lem:dynkin}
Let $D \sub \X$ be open and bounded, and let $f$ be a non-negative function equal to zero in $\overline{D}$. Then
\formula[eq:dynkin]{
 \ex_x f(X({\tau_D})) & = \ex^x \int_0^{\tau_D} \int_{\X \setminus D} \nu(X_t, y) f(y) m(dy) dt
}
for $x \in D$.
\end{lemma}

Using the definition of $\tilde{f}_{8r,s}$ and~\eqref{eq:nu} to substitute $\nu(x_0, y)$ for $\nu(X_t, y)$ in~\eqref{eq:dynkin}, we have
\formula[eq:dynkin:est]{
 \tilde{f}_{8r,s}(x) & \approx \cnu(x_0, 8r, s) M_{s,\infty}(f) \ex_x \tau_{{8r}} .
}
Note that not only we have $M_{s,\infty}(f)$ instead of $M_{6r,\infty}(\tilde{f}_{8r,s})$, but also the constant $\cnu(x_0, 8r, s)$ tends to $1$ as $r \to 0^+$.

\begin{lemma}
\label{lem:ratio:f}
If $0 < 8 r \le s \le \tfrac{R}{4}$, then
\formula{
 \frac{f_{8r,s}(x)}{\tilde{f}_{8r,s}(x)} & \le \cbhi^4 \, \frac{M_{6r,s}(s_{{R/2}})}{1 + M_{s,R/4}(s_{{R/2}})}
}
for $x \in D_{2r}$. If $0 < 16 r \le s \le \tfrac{R}{24}$, then
\formula{
 \frac{f_{8r,s}(x)}{\tilde{f}_{8r,s}(x)} & \ge \cbhi^{-3} \cnu^{-1} \caux^{-3} \, \frac{M_{8r,s}(s_{{R/2}})}{1 + M_{s,R/4}(s_{{R/2}})}
}
for $x \in D_r$.
\end{lemma}

\begin{proof}
By Theorem~\ref{thm:bhi},
\formula{
 f_{8r,s}(x) & \le \cbhi M_{6r,\infty}(f_{8r,s}) \ex_x \tau_{{4r}} , \\
 \tilde{f}_{8r,s}(x) & \ge \cbhi^{-1} M_{6r,\infty}(\tilde{f}_{8r,s}) \ex_x \tau_{{4r}} .
}
Furthermore,
\formula{
 M_{6r,\infty}(f_{8r,s}) & = M_{6r,s}(f_{6r,s}) \le M_{6r,s}(f) , \\
 M_{6r,\infty}(\tilde{f}_{8r,s}) & \ge M_{s,\infty}(\tilde{f}_{8r,s}) = M_{s,\infty}(f) \ge M_{3R/4,\infty}(f) + M_{s,R/4}(f) .
}
Finally, by~\eqref{eq:mf:ms},
\formula{
 M_{6r,s}(f) & \le \cbhi M_{3R/4,\infty}(f) M_{6r,s}(s_{{R/2}}) , \\
 M_{s,R/4}(f) & \ge \cbhi^{-1} M_{3R/4,\infty}(f) M_{s,R/4}(s_{{R/2}}) .
}
We conclude that
\formula{
 \frac{f_{8r,s}(x)}{\tilde{f}_{8r,s}(x)} & \le \cbhi^4 \frac{M_{6r,s}(s_{{R/2}})}{1 + M_{s,R/4}(s_{{R/2}})} \, ,
}
which is the desired upper bound. The lower bound is proved in a somewhat more complicated way. By Theorem~\ref{thm:bhi} and estimate~\eqref{eq:dynkin:est},
\formula{
 f_{8r,s}(x) & \ge \cbhi^{-1} M_{6r,\infty}(f_{8r,s}) \ex_x \tau_{{4r}} , \\
 \tilde{f}_{8r,s}(x) & \le \cnu M_{s,\infty}(f) \ex_x \tau_{{8r}}
}
(we can write $\cnu = \cnu(x_0, 8 r, 16 r)$ in the second inequality because $s \ge 16 r$). By Lemma~\ref{lem:sdr}, $\ex_x \tau_{{8r}} \le \caux \ex_x \tau_{{4r}}$. Furthermore, by Theorem~\ref{thm:bhi} (as in~\eqref{eq:mf:ms}, but with $R$ replaced by $R/3$) and again Lemma~\ref{lem:sdr},
\formula{
 M_{6r,\infty}(f_{8r,s}) & = M_{6r,s}(f_{8r,s}) \ge M_{8r,s}(f_{8r,s}) = M_{8r,s}(f) \\
 & \ge \cbhi^{-1} M_{R/4,\infty}(f) M_{8r,s}(s_{{R/6}}) \\
 & \ge \cbhi^{-1} \caux^{-2} M_{R/4,\infty}(f) M_{8r,s}(s_{{2R/3}}) .
}
On the other hand, by~\eqref{eq:mf:ms},
\formula{
 M_{s,\infty}(f) & = M_{s,R/4}(f) + M_{R/4,\infty}(f) \\
 & \le \cbhi M_{3R/4,\infty}(f) M_{s,R/4}(s_{{R/2}}) + M_{R/4,\infty}(f) \\
 & \le M_{R/4,\infty}(f) (1 + \cbhi M_{s,R/4}(s_{{R/2}})) .
}
We conclude that
\formula{
 \frac{f_{8r,s}(x)}{\tilde{f}_{8r,s}(x)} & \ge \cbhi^{-3} \cnu^{-1} \caux^{-3} \, \frac{M_{8r,s}(s_{{R/2}})}{1 + M_{s,R/4}(s_{{R/2}})} \, ,
}
as desired.
\end{proof}

\subsection{Inaccessible boundary points}
\label{sec:inacc}

Throughout this part we assume that $x_0$ is \emph{inaccessible}, that is,
\formula{
 M_{0,\infty}(s_{R}) = \int_{D_R} \ex_y \tau_{R} \, \nu(x_0, y) m(dy) < \infty .
}
In this case $f_{8r,s}$ and $g_{8r,s}$ turn out to be negligible compared to $\tilde{f}_{8r,s}$ and $\tilde{g}_{8r,s}$ for sufficiently small $r$ and $s$.

Clearly, $M_{0,\infty}(s_{{R/2}}) \le M_{0,\infty}(s_{R}) < \infty$. We remark that by~\eqref{eq:mf:ms},
\formula{
 M_{0,\infty}(f) & = M_{0,R/4}(f) + M_{R/4,\infty}(f) \\
 & \le \cbhi M_{3R/4,\infty}(f) M_{0,R/4}(s_{{R/2}}) + M_{R/4,\infty}(f) < \infty ,
}
and $M_{0,\infty}(g) < \infty$ by the same argument, and hence one can pass to the limit separately in the numerator and the denominator of~\eqref{eq:thm:martin}.

Let $0 < \eps < 1$. By the upper bound in Lemma~\ref{lem:ratio:f}, there is $s = s(\eps) \le \eps R$ such that if $0 < 8r \le s$, then
\formula[eq:ftilde_dominates_f]{
 f_{8r,s}(x) & \le \eps \tilde{f}_{8r,s}(x) , & g_{8r,s}(x) & \le \eps \tilde{g}_{8r,s}(x)
}
for $x \in D_{2r}$. Furthermore, estimate~\eqref{eq:dynkin:est} and assumption \smash{$\lim\limits_{r \to 0^+} \cnu(x_0, r, R) = 1$} imply that there is $r = r(\eps) \le s/8$ such that
\formula[eq:little_clevy]{
 \tilde{f}_{8r,s}(x) & \approx (1 + \eps) \ex_x \tau_{{8r}} M_{s,\infty}(f) , & \tilde{g}_{8r,s}(x) & \approx (1 + \eps) \ex_x \tau_{{8r}} M_{s,\infty}(g)
}
for $x \in D_{8r}$. It follows that
\formula{
 \frac{f(x)}{g(x)} & \le \frac{(1 + \eps) \tilde{f}_{8r,s}(x)}{\tilde{g}_{8r,s}(x)} \le (1 + \eps)^3 \frac{M_{s,\infty}(f)}{M_{s,\infty}(g)}
}
for $x \in D_{2r}$. The lower bound is proved in a similar manner, and we obtain
\formula[eq:accessible_points_inequality]{
(1 + \eps)^{-3} \frac{M_{s,\infty}(f)}{M_{s,\infty}(g)} \le \frac{f(x)}{g(x)} & \le (1 + \eps)^3 \frac{M_{s,\infty}(f)}{M_{s,\infty}(g)}
} 
for $x \in D_{2r}$. Since $\eps$ was arbitrary and $s$ converges to $0$ as $\eps \to 0^+$, we have
\formula{
 \lim_{\substack{x \to x_0 \\ x \in D}} \frac{f(x)}{g(x)} & = \lim_{s \to 0^+} \frac{M_{s,\infty}(f)}{M_{s,\infty}(g)} \, ,
}
and Theorem~\ref{thm:martin} for inaccessible boundary points is proved.

\subsection{Accessible boundary points}
\label{sec:acc}

In the second part of the proof we assume that $x_0$ is \emph{accessible}, that is,
\formula{
 M_{0,\infty}(s_{R}) & = \int_{D_R} \ex_y \tau_{R} \nu(x_0, y) m(dy) = \infty .
}
In this case $f_{8r,s}$ and $g_{8r,s}$ dominate $\tilde{f}_{8r,s}$ and $\tilde{g}_{8r,s}$ for all sufficiently small $r$.

We remark that by~\eqref{eq:mf:ms} and Lemma~\ref{lem:sdr},
\formula{
 M_{0,\infty}(f) & \ge M_{0,R/4}(f) \ge \cbhi^{-1} M_{3R/4,\infty}(f) M_{0,R/4}(s_{{R/2}}) = \infty ,
}
and $M_{0,\infty}(g) = \infty$ by the same argument. In other words, the numerator and the denominator of the right-hand side of~\eqref{eq:thm:martin} diverge to infinity as $r \to 0^+$. In particular, if the limit of $f(x) / g(x)$ in~\eqref{eq:thm:martin} exists, then it is automatically equal to the right-hand side.

Our argument is based on the following standard oscillation reduction lemma. 

\begin{lemma}
\label{lem:ro}
If $0 < 8 r < s < R_0$, then
\formula{
 \expr{\sup_{y \in D_{2r}} - \inf_{y \in D_{2r}}} \frac{f_{8r,s}(y)}{g_{8r,s}(y)} & \le \frac{\cbhi^4 - 1}{\cbhi^4 + 1} \expr{\sup_{y \in D_s} - \inf_{y \in D_s}} \frac{f(y)}{g(y)} .
}
\end{lemma}

\begin{proof}
For simplicity, we denote
\formula{
A & = \sup_{y \in D_s} \frac{f(y)}{g(y)} \, , & B & = \sup_{y \in D_{2r}} \frac{f_{8r,s}(y)}{g_{8r,s}(y)} \, , \\
a & = \inf_{y \in D_s} \frac{f(y)}{g(y)} \, , & b & = \inf_{y \in D_{2r}} \frac{f_{8r,s}(y)}{g_{8r,s}(y)} \, .
}
Since
\formula{
 a g \ind_{{8r,s}} & \le f \ind_{{8r,s}} \le A g \ind_{{8r,s}} ,
}
we clearly have
\formula[eq:ro:0]{
  a g_{8r,s} & \le f_{8r,s} \le A g_{8r,s} .
}
In particular, $a \le b \le B \le A$, and Theorem~\ref{thm:bhi} applies to \emph{everywhere} non-negative functions $f_{8r,s} - a g_{8r,s}$, $A g_{8r,s} - f_{8r,s}$ and $g_{8r,s}$ (note that $f - a g$ and $A g - f$ typically fail to be non-negative everywhere). By~\eqref{eq:bhi},
\formula{
 \sup_{y \in D_{2r}} \frac{f_{8r,s}(y) - a g_{8r,s}(y)}{g_{8r,s}(y)} & \le \cbhi^4 \inf_{y \in D_{2r}} \frac{f_{8r,s}(y) - a g_{8r,s}(y)}{g_{8r,s}(y)} \, , \\
 \sup_{y \in D_{2r}} \frac{A g_{8r,s}(y) - f_{8r,s}(y)}{g_{8r,s}(y)} & \le \cbhi^4 \inf_{y \in D_{2r}} \frac{A g_{8r,s}(y) - f_{8r,s}(y)}{g_{8r,s}(y)} \, .
}
This translates to $B - a \le \cbhi^4 (b - a)$ and $A - b \le \cbhi^4 (A - B)$, and adding the sides of these inequalities leads to the desired inequality
\formula{
 (\cbhi^4 + 1) (B - b) & \le (\cbhi^4 - 1) (A - a).
\qedhere
}
\end{proof}

For continuous processes (in sufficiently regular domains), the above lemma easily yields the assertion of Theorem~\ref{thm:martin}. For processes with jumps one needs to incorporate the non-local parts $\tilde{f}_{8r,s}$ and $\tilde{g}_{8r,s}$ using Lemma~\ref{lem:ratio:f}. As it was remarked in the introduction, this modification was developed in~\cite{bib:b99}, and extended in~\cite{bib:bkk08}.

Let $0 < \eps < 1$ and $0 < s < \tfrac{R}{24}$. By the lower bound in Lemma~\ref{lem:ratio:f}, there is $r = r(\eps, s) \le \tfrac{s}{8}$ such that
\formula[eq:f_dominates_ftilde]{
 \tilde{f}_{8r,s}(x) & \le \eps f_{8r,s}(x) , & \tilde{g}_{8r,s}(x) & \le \eps g_{8r,s}(x)
}
for $x \in D_r$. It follows that
\formula{
 \expr{\sup_{x \in D_r} - \inf_{x \in D_r}} \frac{f(x)}{g(x)} & \le 
 (1 + \eps) \sup_{x \in D_r} \frac{f_{8r,s}(x)}{g_{8r,s}(x)} - \frac{1}{1 + \eps} \inf_{x \in D_r} \frac{f_{8r,s}(x)}{g_{8r,s}(x)} \, .
}
By Lemma~\ref{lem:ro} and the inequality $1 - (1 + \eps)^{-1} \le \eps$,
\formula[eq:supremums]{
 \expr{\sup_{x \in D_r} - \inf_{x \in D_r}} \frac{f(x)}{g(x)} & \le 
 \frac{\cbhi^4 - 1}{\cbhi^4 + 1} \expr{\sup_{x \in D_s} - \inf_{x \in D_s}} \frac{f(x)}{g(x)} + \eps \expr{\sup_{x \in D_r} + \inf_{x \in D_r}} \frac{f_{8r,s}(x)}{g_{8r,s}(x)} \, .
}
Denote by $Q$ the upper limit of the expression in the left-hand side as $r \to 0^+$. Using~\eqref{eq:ro:0} and taking the upper limit of both sides as $s \to 0^+$ leads to
\formula{
 Q & \le \frac{\cbhi^4 - 1}{\cbhi^4 + 1} \, Q + 2 \eps \sup_{x \in D_{R/4}} \frac{f(x)}{g(x)} \, ,
}
that is,
\formula{
 Q & \le \eps (1 + \cbhi^4) \sup_{x \in D_{R/4}} \frac{f(x)}{g(x)} \, .
}
Since $\eps$ is arbitrary, we conclude that $Q = 0$, and the proof of Theorem~\ref{thm:martin} is complete.

%
%

\subsection{Extensions}
\label{subsec:ex}

We first prove the statement contained in Remark~\ref{rem:exit}. Denote $g(x) = \ex_x \tau_{R}$. Then $g$ is not a regular harmonic function in $D_R$, but for every open $U \sub D_R$,
\formula{
 g(x) & = \ex_x \tau_U + \ex_x g(X({\tau_U})) .
}
We interpret $\ex_x \tau_U$ as if it originated from a jump to a distant point (a point at infinity), and we define
\formula{
 M_{r,\infty}(g) & = 1 + \int_{\X \setminus B_r} g(y) \nu(x_0, y) m(dy) , & \tilde{g}_{r,s}(x) & = \ex_x \tau_{r} + \ex_x((g \ind_{{s,\infty}})(X(\tau_{r}))) ;
}
the definitions of $M_{r,s}(g)$ and $g_{r,s}(x)$ for finite $s$ remain unaltered. One can then follow carefully the proof of Theorem~\ref{thm:martin} and see that no changes are required. This shows validity of Remark~\ref{rem:exit}.

In the remaining part of this section we argue that an extension stated in Remark~\ref{rem:uniform} is true: the limit in Theorem~\ref{thm:martin} converges uniformly in $f$ and $g$, and also in $D$, in the sense of~\eqref{eq:uniform}.

We claim that if $0 < q < R_0$ and $\eta > 0$, then there is $p$, which depends only on $q$, $\eta$ and the characteristics of the process $X_t$, such that $0 < p < q$ and
\formula[eq:ro:reduction]{
\frac{\sup_{x \in D_p} (f(x) / g(x))}{\inf_{x \in D_p} (f(x) / g(x))} - 1 & \le \eta + \frac{\cbhi^4 - 1}{\cbhi^4 + 1} \expr{\frac{\sup_{x \in D_q} (f(x) / g(x))}{\inf_{x \in D_q} (f(x) / g(x))} - 1} 
}
for all open sets $D$ and all functions $f$ and $g$ as in Theorem~\ref{thm:martin} (this estimate is very similar to~\eqref{eq:supremums}). By considering the supremum of both sides of~\eqref{eq:ro:reduction} over all $f$, $g$ and $D$, and then taking the upper limit as $q \to 0^+$, we obtain the desired result:
\formula{
 \limsup_{r \to 0^+} \sup_{D,f,g} \expr{\frac{\sup_{x \in D_r} (f(x) / g(x))}{\inf_{x \in D_r} (f(x) / g(x))} - 1} & \le \frac{\eta (1 + \cbhi^4)}{2}
}
for arbitrary $\eta > 0$. Therefore, it remains to prove~\eqref{eq:ro:reduction}.

Let $0 < q < \tfrac{1}{24} R_0$ and $\eta > 0$. We consider two additional parameters $\delta, N > 0$; the actual values of $\delta$ (small real) and $N$ (large integer) are to be specified at the end of the argument. By the assumption $\lim_{r \to 0^+} \cnu(x_0, r, R) = 1$ one can construct a decreasing sequence of radii $a_0, a_1, \dots, a_N$ so that $a_0$ is the input radius $q$, $\tfrac{1}{8} a_N$ will be the output radius $p$, and we have $16 a_{n+1} < a_n$ and $\cnu(x_0, 8 a_{n+1}, a_n) \le 1 + \delta$ for all $n = 0, 1, \dots, N - 1$.

Following~\cite{bib:bkk08}, we consider two scenarios. Suppose first that for some $n$ we have
\formula[eq:lucky]{
 M_{a_{n+1},a_n}(s_{R/2}) & \le \delta (1 + M_{a_n,R/4}(s_{R/2})) .
}
Then the argument is fairly simple: as in Section~\ref{sec:inacc}, by Lemma~\ref{lem:ratio:f} we have the inequality~\eqref{eq:ftilde_dominates_f} with $r = a_{n+1}$, $s = a_n$ and $\eps = \cbhi^4 \delta$. Since $\cnu(x_0, 8 a_{n+1}, a_n) \le 1 + \delta$, the estimate~\eqref{eq:little_clevy} holds with $r = a_{n+1}$, $s = a_n$ and $\eps = \delta$. This implies~\eqref{eq:accessible_points_inequality} (with $s = a_n$, $x \in D_{2 a_{n+1}}$ and $\eps = \cbhi^4 \delta$), and in particular the left-hand side of~\eqref{eq:ro:reduction} does not exceed $(1 + \cbhi^4 \delta)^6 - 1$. Estimate~\eqref{eq:ro:reduction} follows with $p = a_{n+1}$, provided that $(1 + \cbhi^4 \delta)^6 - 1 \le \eta$. We choose $\delta$ small enough, so that this inequality is satisfied.

In the other scenario, for each $n$ the converse of~\eqref{eq:lucky} holds. Summing up these inequalities for $n = 0, 1, \dots, N - 1$ we obtain
\formula{
 M_{a_N,a_0}(s_{R/2}) & \ge N \delta (1 + M_{a_0,R/4}(s_{R/2})) ,
}
and we argue as in Section~\ref{sec:acc}. Again by Lemma~\ref{lem:ratio:f}, we have~\eqref{eq:f_dominates_ftilde} with $r = \tfrac{1}{8} a_N$, $s = a_0$ and $\eps = \cbhi^3 \cnu \caux^3 (N \delta)^{-1}$. Inequality~\eqref{eq:supremums} follows. Dividing both sides of it by $\inf_{x \in D_r} (f(x) / g(x))$ and using monotonicity of this expression in $r$, we obtain~\eqref{eq:ro:reduction} for $p = \tfrac{1}{8} a_N$, provided that $\eps (\cbhi + 1) \le \eta$. Since $\delta$ is now fixed, we may choose $N$ large enough, so that this condition is satisfied. This completes the proof of the extension described in Remark~\ref{rem:uniform}.

%
%

\subsection{Martin representation}
\label{subsec:repr}

In this section we prove Theorem~\ref{thm:repr}. We assume that the assumptions of Theorem~\ref{thm:martin} are satisfied in a uniform way for all $x_0 \in \overline{D}$.

We note one important property of the Green function: if $U$ is an open subset of $D$, then
\formula[eq:green]{
 G_D(x, y) = \ex_x G_D(X(\tau_U), y) + G_U(x, y)
}
(where, as usual, we assume that $G_U(x, y) = 0$ whenever $x \notin U$ or $y \notin U$). In particular, $G_D(x, y)$ is a regular harmonic function in $D \setminus \overline{B}(y, r)$ for every $r > 0$. By a duality argument, $G_D(x, y)$ is a regular \emph{co-harmonic} function in $D \setminus \overline{B}(x, r)$ for every $r > 0$.

\begin{proof}[Proof of Theorem~\ref{thm:repr}\ref{it:repr:1}]
The assumptions are completely symmetric under duality, and hence we may apply Theorem~\ref{thm:martin} to both harmonic and co-harmonic functions. In particular, as already remarked before the statement of Theorem~\ref{thm:repr}, the Martin kernel, defined as the boundary limit of co-harmonic functions
\formula{
 M_D(x, z) & = \lim_{\substack{y \to z \\ x \in D}} \frac{G_D(x, y)}{G_D(\tilde{x}, y)} \, ,
}
exists for all boundary points $z \in D$ (here and below $\tilde{x} \in D$ is a fixed reference point). In other words, the Martin boundary coincides with the Euclidean boundary.
\end{proof}

The representation given in part~\ref{it:repr:4} essentially follows now from the general theory of Martin boundary, together with some ideas developed in~\cite{bib:bkk08}. For simplicity, in the remaining part of the proof we simply write that a function is harmonic when we refer to harmonicity in $D$.

\begin{proof}[Proof of Theorem~\ref{thm:repr}\ref{it:repr:2}]
Following the proof of Theorem~2 in~\cite{bib:bkk08}, we find that $M_D(x, x_0)$ is a harmonic function with respect to $x$ if and only if $x_0$ is accessible. Indeed, for an inaccessible boundary point $x_0$ we have, by~\eqref{eq:thm:martin} in Theorem~\ref{thm:martin},
\formula{
 M_D(x, x_0) & = C \int_D \nu(y, x_0) G_D(x, y) m(dy)
}
for $C = (\int_D \nu(y, x_0) G_D(\tilde{x}, y) m(dy))^{-1} > 0$, and so the Martin kernel is not harmonic (to see this, simply use~\eqref{eq:green} and Fubini). On the other hand, if $x_0$ is accessible and $R > 0$, then
\formula[eq:martin:harm]{
 \ex_x M_D(X(\tau_{D \setminus \overline{B}(x_0, R)}), x_0) & = \ex_x \lim_{\substack{y \to x_0 \\ y \in D}} \frac{G_D(X(\tau_{D \setminus \overline{B}(x_0, R)}), y)}{G_D(\tilde{x}, y)} .
}
Recall that $G_D(x, y)$ is a regular harmonic function of $x \in D \setminus \overline{B}(x_0, R)$ when $y \in B(x_0, R)$. By Fatou's lemma,
\formula[eq:martin:subharm]{
 \ex_x M_D(X(\tau_{D \setminus \overline{B}(x_0, R)}), x_0) & \le M_D(x, x_0) ,
}
and we claim that in fact equality holds, that is, we can exchange the limit with the expectation in~\eqref{eq:martin:harm}. By Vitali's convergence theorem, it suffices to prove that that the ratio in the right-hand side of~\eqref{eq:martin:harm} is a uniformly integrable family of random variables for $y \in D \cap B(x_0, r)$ for some $r > 0$. The argument is exactly the same as in the proof of formula~(77) in~\cite{bib:bkk08}; for the convenience of the reader, we repeat it below.

Assume that $0 < 8 r < R$ and that $x, \tilde{x} \notin D \cap B(x_0, R)$. We will first prove that
\formula[eq:martin:supest]{
 \sup_{\substack{y \in D \cap B(x_0, r) \\ z \in D \setminus B(x_0, 4 r)}} \frac{G_D(z, y)}{G_D(\tilde{x}, y)} < \infty .
}
By the boundary Harnack inequality (Theorem~\ref{thm:bhi}) applied to $G_D(z, \cdot)$ and $G_D(\tilde{x}, \cdot)$, it suffices to consider a fixed $y \in D \cap B(x_0, r)$, that is, to show that $G_D(\cdot, y)$ is bounded in $D \setminus B(x_0, 4 r)$. This is relatively simple, but somewhat technical. Denote $D_1 = D \cap B(x_0, r)$, $D_2 = D \cap B(x_0, r)$, $D_4 = D \cap \overline{B}(x_0, 4 r)$ and $D' = D \setminus \overline{B}(x_0, 4 r)$. By Dynkin's formula~\eqref{eq:dynkin},
\formula{
 & \ex_z \bigl(G_D(X(\tau_{D'}), y) \ind_{D_2}(X(\tau_{D'}))\bigr) & \\
 & \hspace*{3em} \le \biggl(\sup_{\substack{v \in D' \\ w \in D_2}} \nu(v, w)\biggr) \int_{D'} G_{D'}(z, v) m(dv) \int_{D_2} G_D(w, y) m(dw) .
}
The supremum is finite by Assumption~\ref{asm:levy} and boundedness of $D$, and the integrals in the right-hand side are bounded by $\sup_{u \in D} \ex_u \tau_D$ and $\sup_{u \in D} \hat{\ex}_u \hat{\tau}_D$, respectively. Furthermore,
\formula{
 \ex_z \bigl(G_D(X(\tau_{D'}), y) \ind_{D_4 \setminus D_2}(X(\tau_{D'}))\bigr) & \le \sup_{\substack{v \in D_4 \setminus D_2 \\ w \in D_1}} G_D(v, w) ,
}
and the right-hand side is finite by Assumption~\ref{asm:green}. By adding the sides of these two bounds and using harmonicity of the Green function, we complete the proof of~\eqref{eq:martin:supest}.

On the other hand, if we denote $D'' = D \setminus \overline{B}(x_0, 8 r)$ and $D''' = D \setminus \overline{B}(x_0, R)$, then, again by Lemma~\ref{lem:dynkin},
\formula{
 & \ex_x \bigl(G_D(X(\tau_{D'''}), y) \ind_{D_4}(X(\tau_{D'''}))\bigr) \\
 & \hspace*{3em} \le \cnu \expr{\int_{D'''} \nu(x_0, v) G_{D'''}(x, v) m(dv)} \expr{\int_{D_4} G_D(w, y) m(dw)} ,
}
and, in a similar way,
\formula{
 G_D(\tilde{x}, y) & \ge \ex_{\tilde{x}} \bigl(G_D(X(\tau_{D''}), y) \ind_{D_4}(X(\tau_{D''}))\bigr) \\
 & \ge \cnu^{-1} \expr{\int_{D''} \nu(x_0, v) G_{D''}(\tilde{x}, v) m(dv)} \expr{\int_{D_4} G_D(w, y) m(dw)} .
}
It follows that
\formula{
 & \ex_x \expr{ \frac{G_D(X(\tau_{D \setminus \overline{B}(x_0, R)}), y)}{G_D(\tilde{x}, y)} \, \ind_{D \cap B(x_0, 4 r)}(X(\tau_{D \setminus B(x_0, R)}))} \\
 & \hspace*{11em} \le C \expr{\int_{D \setminus B(x_0, 8 r)} \nu(x_0, v) G_{D \setminus B(x_0, 8 r)}(\tilde{x}, v) m(dv)}^{-1} ,
}
where $C$ does not depend on (sufficiently small) $r > 0$ and $y \in D \cap B(x_0, r)$. Recall that $G_{D \setminus B(x_0, 8 r)}(\tilde{x}, v)$ increases to $G_D(\tilde{x}, v)$ (because the corresponding exit times $\tau_{D \setminus B(x_0, 8 r)}$ increase to $\tau_D$). By monotone convergence, the right-hand side converges to zero as $r \to 0^+$. Together with~\eqref{eq:martin:supest}, this completes the proof of uniform integrability of the right-hand side of~\eqref{eq:martin:harm}.

Part~\ref{it:repr:2} follows, and in addition we see that for accessible boundary points $z$, the Martin kernel $M_D(x, z)$ is a regular harmonic function in $D \setminus B(z, r)$ for every $r > 0$.
\end{proof}

In order to apply the general theory of Martin boundary, we need to prove that the Green operator, which maps a measurable function $f(x)$ to $G_D f(x) = \int_D G_D(x, y) f(y) m(dy)$, takes bounded functions into continuous ones. Let $f$ be a bounded function on $D$, $x_0 \in D$ and $\eps > 0$. Clearly, $|G_D f(x)| \le \|f\| \ex_x \tau_D$ for $x \in D$, so that $G_D f$ is bounded. Let $r > 0$ be small enough, so that $\ex_x \tau_{B(x_0, r)} < \eps$ for $x \in B(x_0, r)$. By~\eqref{eq:green},
\formula{
 G_D f(x) & = \ex_x G_D f(X(\tau_{B(x_0, r)})) + G_{B(x_0, r)} f(x) .
}
The first term is continuous in $B(x_0, r)$ by Theorem~\ref{thm:martin} (see Remark~\ref{rem:continuity}). The other one is bounded by $\eps \|f\|$, an arbitrarily small number. Therefore, $G_D f$ is continuous at $x_0$.

The general theory of Martin boundary tells us now that if $f$ satisfies the assumptions of Theorem~\ref{thm:repr} and $f$ is equal to zero in the complement of $D$, then
\formula[eq:repr:m]{
 f(x) & = \int_{\partial_m D} M_D(x, z) \mu(dz)
}
for some measure $\mu$ on the set of accessible boundary points $\partial_m D$, see Theorem~14.8 in~\cite{bib:cw05}. Furthermore, if we show that for every $z \in \partial_m D$, $M_D(x, z)$ is a \emph{minimal} harmonic function with respect to $x$, then the measure $\mu$ in the above representation is unique. Minimality of $M_D(x, z)$ is proved as in the final part of the proof of Lemma~14 in~\cite{bib:bkk08}.

\begin{proof}[Proof of Theorem~\ref{thm:repr}\ref{it:repr:3}]
Suppose that $f$ is harmonic, $0 \le f(x) \le M_D(x, x_0)$ for all $x \in \X$ (in particular, $f(x) = 0$ for $x \in \X \setminus D$) and that the measure $\mu$ in representation~\eqref{eq:repr:m} is zero on $\partial_m D \cap B(x_0, 4 r)$ for some $r > 0$. Our goal is to prove that $f$ is identically zero. This will imply that if $f$ is harmonic and $0 \le f(x) \le M_D(x, x_0)$ for all $x \in \X$, then the measure $\mu$ in representation~\eqref{eq:repr:m} is concentrated in $\{x_0\}$, and thus $M_D(x, x_0)$ is a minimal harmonic function.

For every $z \in \partial_m D \setminus B(x_0, 4 r)$, $M_D(x, z)$ is a regular harmonic function in $D \cap B(x_0, 3 r)$. Hence, by Fubini, $f$ also has this property. Furthermore, by the boundary Harnack inequality (Theorem~\ref{thm:bhi}), $f$ is bounded on $D \cap B(x_0, 2 r)$.

On the other hand, since $f(x) \le M_D(x, x_0)$, one easily finds that $f$ is also a regular harmonic function in $D \setminus B(x_0, r)$. This is exactly the same argument as in Lemma~9 in~\cite{bib:bkk08}; for the convenience of the reader, we provide the details at the end of this section. In particular, since $f$ is bounded in $D \cap B(x_0, 2 r)$, it is bounded on $D$.

A sweeping argument, which is a simplified version of Lemma~10 in~\cite{bib:bkk08}, proves then that $f$ is a regular harmonic function in $D$: Let $\sigma_n$ be the sequence of consecutive exit times from alternately $D \cap B(x_0, 4 r)$ and $D \setminus B(x_0, r)$. That is, $\sigma_0 = 0$ and $\sigma_{n+1} = \sigma_n + \tau_V \circ \thet_{\sigma_n}$, where $V = D \cap B(x_0, 4 r)$ when $n$ is even and $V = D \setminus B(x_0, r)$ when $n$ is odd (and $\thet_\tau$ is the shift operator).

Clearly, $\sigma_n \le \tau_D < \infty$. Since $\sigma_n$ is increasing, by quasi-left continuity, $X(\sigma_n)$ has a limit as $n \to \infty$. Therefore, it is impossible that $\sigma_n < \tau_D$ for infinitely many $n$. It follows that with probability one, eventually $\sigma_n = \tau_D$.

Since $f(x) = \ex_x f(X(\sigma_n))$ and $f$ is bounded, by dominated convergence we have $f(x) = \ex_x f(X(\tau_D)) = 0$, as desired.
\end{proof}

We have thus proved the representation~\eqref{eq:repr:m} for harmonic functions $f$ which are zero in the complement of $D$. The general case is handled as in Lemma~13 in~\cite{bib:bkk08}. 

\begin{proof}[Proof of Theorem~\ref{thm:repr}\ref{it:repr:4}]
Let $D_n$ be an ascending sequence of open sets such that $\overline{D_n} \sub D$ and $\bigcup_{n = 1}^\infty D_n = D$. Then, by Lemma~\ref{lem:dynkin},
\formula{
 f(x) & = \ex_x f(X(\tau_{D_n})) \ge \int_{\X \setminus D} \expr{\int_{D_n} G_{D_n}(x, y) \nu(y, z) m(dy)} f(z) m(dz) .
}
The integrand in the right-hand side increases as $n \to \infty$, and therefore by monotone convergence,
\formula[eq:repr:p]{
 f(x) & \ge \int_{\X \setminus D} \expr{\int_D G_D(x, y) \nu(y, z) m(dy)} f(z) m(dz) .
}
Let $g(x)$ be equal to the right-hand side of~\eqref{eq:repr:p} for $x \in D$, and to $f(x)$ for $x \in \X \setminus D$. From Lemma~\ref{lem:dynkin} and the property~\eqref{eq:green} of the Green function it follows easily that $g$ is harmonic: if $U$ is open and $\overline{U} \sub D$, then
\formula{
 \ex_x g(X(\tau_U)) & = \ex_x (f \ind_{\X \setminus D})(X(\tau_U)) + \ex_x \int_{\X \setminus D} \expr{\int_D G_D(X(\tau_U), y) \nu(y, z) m(dy)} f(z) m(dz) \\
 & = \int_{\X \setminus D} \expr{\int_D \bigl(G_U(x, y) + \ex_x G_D(X(\tau_U), y)\bigr) \nu(y, z) m(dy)} f(z) m(dz) = g(x) .
}
Therefore, $f - g$ is a non-negative harmonic function which is equal to zero in $\X \setminus D$, and so it has a unique representation~\eqref{eq:repr:m}.

Finally, the outer integral in~\eqref{eq:repr:p} is finite, and so points at which the inner integral is infinite cannot contribute to the integral. It follows that we can change the outer integral to an integral over $\X \setminus (D \cup \partial_m D)$. The proof of~\eqref{eq:repr} is complete.
\end{proof}

\begin{proof}[Proof of Theorem~\ref{thm:repr}\ref{it:repr:5}]
By the boundary Harnack inequality, if the right-hand side of~\eqref{eq:repr} is finite at some $x \in D$, it is finite everywhere in $D$. Indeed, let $f$ be given by~\eqref{eq:repr}. If $f(x) = \infty$ for some $x \in D$, by Theorem~\ref{thm:bhi} $f$ is infinite at every point of a ball $B(x, r)$ contained in $D$. If $y \in D$, then again using Theorem~\ref{thm:bhi} (for a ball centred at $y$), $f$ is infinite at $y$.

Finally, harmonicity of the right-hand side of~\eqref{eq:repr}, whenever it is finite, follows from Lemma~\ref{lem:dynkin}, property~\eqref{eq:green} of the Green function, and harmonicity of the Martin kernel.
\end{proof}

At the end of this section, we present the proof of Lemma~9 in~\cite{bib:bkk08}, adapted to our setting. This result was used in the proof of Theorem~\ref{thm:repr}\ref{it:repr:3}.

\begin{lemma}[{Lemma~9 in~\cite{bib:bkk08}}]
Let $U$ and $D$ be open subsets of $\X$ such that $U \sub D$. If $0 \le f(x) \le g(x)$ for all $x \in \X$, $f$ and $g$ are harmonic in $D$, $g$ is a regular harmonic function in $U$ and $g(x) = 0$ for $x \in \X \setminus D$, then $f$ is a regular harmonic function in $U$.
\end{lemma}

\begin{proof}
Let $D_n$ be an ascending sequence of open sets such that $\overline{D_n} \sub D$ and $\bigcup_{n = 1}^\infty D_n = D$, and let $U_n = U \cap D_n$. Then $\tau_{U_n}$ increases to $\tau_U$, and, by quasi-left continuity, $X(\tau_{U_n})$ converges to $X(\tau_U)$ with probability one. It follows that if $X(\tau_U) \in D \setminus U$, then eventually $\tau_{U_n} = \tau_U$ for $n$ large enough up to an event of probability zero. Hence,
\formula{
 \lim_{n \to \infty} \ex_x (g \ind_{D \setminus U}) (X(\tau_{U_n})) = \ex_x (g \ind_{D \setminus U}) (X(\tau_U)) = g(x) .
}
Therefore,
\formula{
 \ex_x (f \ind_{U \setminus U_n})(X(\tau_{U_n})) & \le \ex_x (g \ind_{U \setminus U_n})(X(\tau_{U_n})) = g(x) - \ex_x (g \ind_{D \setminus U}) X(\tau_{U_n}))
}
converges to zero as $n \to \infty$. It follows that
\formula{
 f(x) & = \lim_{n \to \infty} \ex_x f(X(\tau_{U_n})) = \lim_{n \to \infty} \ex_x (f \ind_{D \setminus U})(X(\tau_{U_n})) \\
 & = \ex_x (f \ind_{D \setminus U})(X(\tau_U)) = \ex_x f(X(\tau_U)) ,
}
as desired.
\end{proof}

%
%

%
%

\end{document}